\documentclass{amsart}
\usepackage{amssymb,euscript,tikz,units}
\usepackage{verbatim}

\newcommand{\ML}{\ensuremath{\mathcal{ML}}}

\newcommand{\PML}{\ensuremath{\mathcal{PML}}}

\renewcommand{\H}{\mathbb{H}}
\newcommand{\Z}{\mathbb{Z}}
\newcommand{\calT}{\mathcal{T}}
\newcommand{\T}{\calT}
\newcommand{\euS}{\EuScript{S}}
\newcommand{\euB}{\EuScript{B}}
\newcommand{\R}{\mathbb{R}}

\newcommand{\CC}{\mathcal{C}}
\DeclareMathOperator{\Mod}{Mod}
\DeclareMathOperator{\Diff}{Diff}

\DeclareMathOperator{\I}{i}
\newcommand{\di}{d_{\rm i}}
\newcommand{\dc}{d_S}
\newcommand{\Sph}[2]{\mathcal{S}_{#1}(#2)}
\newcommand{\Sphbar}[2]{\overline{\mathcal{S}_{#1}(#2)}}

\newcommand{\short}[3]{P_{#2}^{#1}(#3)}

\theoremstyle{plain}
\newtheorem{theorem}{Theorem}
\newtheorem{corollary}[theorem]{Corollary}
\newtheorem{proposition}[theorem]{Proposition}
\newtheorem{lemma}[theorem]{Lemma}

\newtheorem*{blurb}{Theorem}

\theoremstyle{definition}
\newtheorem{definition}[theorem]{Definition}

\newtheorem{remark}[theorem]{Remark}
\newtheorem{example}[theorem]{Example}
\begin{document}

\title{Spheres in the curve complex}
\author[Dowdall Duchin Masur]{Spencer Dowdall, Moon Duchin, and Howard Masur}
\maketitle

\section{Introduction}

The curve graph (or curve complex) $\CC(S)$ associated to a surface $S$ of finite type   is a locally infinite combinatorial object
that encodes topological information about the surface through intersection patterns of simple closed curves.
It is known to be $\delta$-hyperbolic 
~\cite{MM1}, a property that is often described by saying that a space is ``coarsely a tree."  To be precise, there exists $\delta$ such that for  any geodesic triangle, each side is in the $\delta$-neighborhood of the union of the other two sides.  
In this note, we will investigate the finer metric properties of the curve graph by considering the geometry of 
spheres; specifically, we will study the average distance between pairs of points on $\Sph r{\alpha}$,
the sphere of radius $r$ centered at $\alpha$.  To make sense of the idea of averaging, we will develop a definition of null and generic sets 
in \S\ref{nullness} that is compatible with the topological structure of the curve graph.  

Given a  family of
probability measures $\mu_r$ on the spheres $\Sph rx$ in a metric space $(X,d)$, 
let $E(X)=E(X,x,d,\{\mu_r\})$ be  the normalized average distance between points on large spheres:
$$E(X):=\lim_{r\to\infty} \frac 1r \int_{\Sph rx \times \Sph rx} d(y,z) \ \ d\mu_r(y) d\mu_r(z),$$
if the limit exists.  For finitely generated groups with their Cayley graphs, 
or more generally for locally finite graphs, 
we can study averages with respect to counting measure because the spheres are finite sets.
It is shown in \cite{DLM2} that non-elementary hyperbolic groups all have 
$E(G,S)=2$ for any finite generating set $S$; this is also the case in the hyperbolic space $\H^n$ of 
any dimension endowed with the natural measure on spheres.
By contrast, $E(\R^n)<2$ and  
$E(\Z^n,S)<2$ for all $n$ and $S$, with nontrivial dependence on $S$.
In particular this shows that $E(X)$ varies under quasi-isometry, so 
it is a fine and not a coarse asymptotic statistic.
Note that $\delta$-hyperbolicity itself (without an assumption of homogeneity) does
not imply $E=2$:  even for locally finite trees, one can get any value $0\le E\le 2$.

In \cite{DDM1}, we show that for Teichm\"uller space with the Teichm\"uller metric and various visual measures on 
spheres, $E(\T(S))=2$.
Here, we show something stronger for the curve graph.
\begin{blurb} With respect 
to our notion of genericity for spheres in the curve graph, almost every pair of points on $\Sph r\alpha$ is at distance exactly $2r$ apart.  
\end{blurb}
This holds for every $r$ and is certainly stronger than 
saying that the average distance is asymptotic to $2r$, so we can write $E(\CC(S))=2$.  
That is, suppose we start with $\alpha$ and a pair of curves $\beta$ and $\gamma$ such that the shortest path in the curve graph from $\alpha$ to either $\beta$ or $\gamma$ has length $r$. Then, almost surely, there is no more efficient way to connect them with each other than to travel through 
$\alpha$.
This result tells us that, even though the space $\CC(S)$ is far from uniquely geodesic,
the concatenation of two geodesic segments of length $r$ that share a common endpoint is almost 
always itself geodesic.
In this sense the curve graph is ``even more hyperbolic than a tree."

Of course, the meaningfulness of this result depends on the notion of genericity. Lustig and Moriah \cite{LM2} have introduced a very natural definition for genericity in $\CC(S)$ which uses the topology and measure class of 
$\PML(S)$. We identify the sphere of radius $1$ in $\CC(S)$ 
with a lower-complexity curve complex, so that genericity can be defined in the same way. 
We then extend  to spheres of larger radius in a ``visual'' manner; see Definition~\ref{def:nullinspheres}. While the Lustig--Moriah definition gives content to statements about typical curves on $S$, our  notion of genericity on spheres enables us to talk about typical properties of high-distance curves on $S$. 

\section{Background}

We fix a topological surface $S=S_{g,n}$ with genus $g$ and $n$ punctures, 
and let $h=6g-6+2n$.
Let $\euS$ be the set of homotopy classes of essential nonperipheral 
simple closed curves on $S$.  From now on, a {\em curve} will mean an element 
of $\euS$.  
Next we define the {\em curve graph} $\CC(S)$: The vertex set of $\CC(S)$ is $\euS$. In the case that $h> 2$, two curves are joined by an edge if they are disjointly realizable. In the case of $S_{1,1}$ we join two vertices if the curves intersect once, and in the case of $S_{0,4}$ 
two vertices are joined by an edge if the curves intersect twice.
In each of these cases, $\CC(S)$ is endowed with the standard path metric, denoted
$\dc(\alpha,\beta)$.

For $\alpha\in\euS$, we write $S_\alpha$ to denote the lower-complexity punctured (possibly disconnected) surface
obtained by cutting open $S$ along $\alpha$.  
 Note that $\CC(S_\alpha)$ can be realized as the subgraph of $\CC(S)$ consisting of neighbors
of $\alpha$---that is, it is identified with the sphere $\Sph 1\alpha \subset\CC(S)$.

Recall that a {\em measured lamination} on $S$, given a hyperbolic structure,  is a foliation of a closed subset of $S$ by geodesics, together with a
measure on transversals that is invariant under holonomy along the leaves of  the lamination.
We will use $\ML(S)$ to denote the space of {\em  measured laminations} on $S$. 
Let $\Mod(S):=\pi_0(\Diff^+\!(S))$ denote the {\em mapping class group} of $S$; it acts on $\CC(S)$ and on $\ML(S)$.  The latter has  a natural $\Mod(S)$-invariant measure $\mu$.
(This is the 
Lebesgue measure associated to the  piecewise linear structure induced on $\ML(S)$ by train track neighborhoods.)
The space of {\em projective measured laminations}  
$\PML(S)$ is obtained by identifying laminations whose transverse measures differ only by a scalar multiple; 
it is endowed with the topology of a sphere of dimension $h-1$ and it inherits a natural 
$\Mod(S)$-invariant measure class, which we will again denote by $\mu$.

A {\em train track} on $S$ is a finite collection of disjointly embedded arcs, called {\em branches}, meeting at vertices, called {\em switches},
such that the branches are $C^1$ away from switches and have well-defined tangents at the switches.  (There are also 
non-degeneracy conditions for switches and topological conditions on the complement; see \cite{PH} for details.)
Train tracks exist on every surface with $h>0$, and we say that a lamination $F$ is {\em carried} by a train track $\tau$
if there is a map $\phi: S\to S$ isotopic to the identity with $\phi(F)\subset \tau$ . Via this carrying relation,
measured laminations that are carried by $\tau$ correspond to choices 
of weights on the branches of $\tau$ that satisfy switch conditions.

A connected proper subsurface $V$ of $S$ is {\em essential}  if all components of $\partial V$ are essential; i.e., they are homotopically nontrivial and not isotopic to a puncture.

\begin{definition}  
Consider a non-annular essential subsurface $V$.
The \emph{subsurface projection} $\pi_V$ is a coarsely well-defined map $\CC(S)\to \CC(V)$ defined as follows.
Realize $\beta\in \euS$ and 
$\partial V$ as geodesics (in any hyperbolic metric on $S$). If $\beta\subset V$, let $\pi_V(\beta)=\beta$. 
If $\beta$ is disjoint from $V$, then $\pi_V(\beta)$ is undefined. Otherwise, $\beta\cap V$ is a disjoint union of finitely many homotopy classes 
of arcs with endpoints on $\partial V$, and we obtain $\pi_V(\beta)$ by choosing any arc and performing a surgery along $\partial V$ to create 
a simple closed curve contained in $V$.  All possible ways to do this form  a non-empty subset of the curve complex $\CC(V)$ with 
 uniformly bounded diameter. We can denote by $d_V(\alpha, \beta)$ the diameter in $\CC(V)$ of $\pi_V(\alpha)\cup \pi_V(\beta)$. If $\alpha$ and $\beta$ are disjoint and both intersect $V$ then $d_V(\alpha,\beta)\leq 4$. 
\end{definition}

There is a well-defined inclusion $\euS\hookrightarrow \PML(S)$ whose image is dense 
and we will identify $\euS$ with its image under that map.
The {\em supporting subsurface} of a lamination  is the subsurface filled by $F$.
We will denote the geometric intersection number on $\euS$ by $\I(\alpha,\beta)$, and we recall that it
has a  well-defined extension to $\ML(S)$.  On $\PML(S)$, we can thus talk about whether or not $\I(F,G)=0$.
Then $\PML(S_\alpha)$ can be identified with the subset of
$\PML(S)$ consisting of those laminations $F$ for which $\I(F,\alpha)=0$.
If $\alpha$ is nonseparating, then $S_\alpha$ has complexity $h-2$; if $\alpha$ is separating, then $S_\alpha$
consists of two surfaces with complexity $h_1+h_2=h-2$.  In that case we consider $\PML(S_\alpha)$ as a product of the corresponding spaces for the two components.  In either case 
we see that $\PML(S_\alpha)$ has positive codimension in $\PML(S)$.

\begin{definition}
Given a group $G$ that acts on a space $X$ with Borel algebra $\euB(X)$, a {\em $G$-invariant mean} on $X$ 
is a function $\sigma: \euB(X)\to [0,1]$ such that 
\begin{itemize}
\item $\sigma(\emptyset)=0$ and $\sigma(X)=1$ \ ;
\item if $B_1,\ldots B_N\in\euB(X)$ are pairwise disjoint then $\sigma(\bigsqcup B_i)=\sum \sigma(B_i)$ \ ; \quad and 
\item $\sigma(B)=\sigma(gB)$ for all $B\in\euB(X)$ and all $g\in G$.
\end{itemize}
\end{definition}

Note that invariant means are only required to be finitely additive, while measures must be countably additive.

\begin{proposition}
There is no  $\Mod(S)$-invariant mean on $\euS$ or on $\PML$. 
\end{proposition}

\begin{proof}
One can choose a finite number $T=T(S)$ of train tracks $\tau_1,\ldots, \tau_T$ such that (1) every curve is carried by $\tau_i$
for some $i$ and (2) the set of curves that are carried by $\tau_i$ and have positive weights on all branches of $\tau_i$ is disjoint 
from the corresponding set of curves carried by $\tau_j$.
Pairs of attracting and repelling fixed points of pseudo-Anosov diffeomorphisms are dense in $\PML\times \PML$.

If there is an invariant mean $\sigma$ on $\euS$, for some $i$ the set $B_i$ of curves carried by $\tau_i$ satisfies 
$a_i=\sigma(B_i)>0$.
Choose $N$ such that $N>1/a_i$ for such an index $i$. 
Find $N$ pseudo-Anosovs $\psi_k$ with distinct attracting fixed points at laminations carried by $\tau_i$ 
and repelling fixed points carried 
by some $\tau_j$ with $j\neq i$. Disjoint neighborhoods around these attracting fixed points may be chosen such that all curves in each neighborhood are carried by $\tau_i$ and have positive weights on all branches of $\tau_i$; likewise for the repelling fixed points.  
Raising each $\psi_k$ to a high enough power $m_k$ we can conclude that  the curves carried by $\psi_k^{m_k}(\tau_i)$ 
are disjoint from the set of curves carried by $\psi_l^{m_l}(\tau_i)$ for $l\neq k$.  
By invariance, each of those sets has $\sigma$-mean $a_i$. Adding $N$ of these we find that $\sigma(\euS)>1$, a  
contradiction.
\end{proof}

\begin{corollary}
There is no $\Mod(S)$-invariant Borel probability measure $\mu$ on $\PML$.
\end{corollary}

\section{Genericity in the curve complex}\label{nullness}

In the paper \cite{LM2}, Lustig--Moriah give the following notion of genericity.

\begin{definition}
Let $X$ be a topological space, provided with a Borel 
measure or measure class $\mu$. Let $Y\subset X$ be a (possibly countable) subset with $\mu(\overline Y) \neq 0$.
Then the set $A\subset Y$ is called {\em generic in $Y$} (or simply 
generic, if $Y = X$) if  
$\mu(\overline{Y\setminus A}) = 0$. (Here closures are taken in $X$.)   
On the other hand $A$ is called {\em null} in $Y$ if $\mu(\overline A)=0$. 
\end{definition}

We can extend the definitions to products as follows.  Given  $E\subset Y\times Y$ and $a\in Y$ let 
$E(a):=\{b\in Y : (a,b)\in E ~\hbox{or}~ (b,a)\in E\}$.

\begin{definition}\label{def:nullinspheres}
$E$ is {\em null} in $Y\times Y$ if  $\{a\in Y : E(a) ~ \text{not null in}~Y\}$ is null.
\end{definition}

This definition for products corresponds to Fubini's theorem:  the set of points with non-null fibers must be null. 

We will focus on the case that $X=\PML(S)$ for any surface $S$ and $Y=\euS(S)$ is the set of simple closed curves.
Several examples and observations can be made immediately to illustrate that this notion is topologically interesting.
\begin{itemize}
\item  Nullness in $\euS$ is preserved by:  acting by $\Mod(S)$, passing to subsets, and finite unions.
A set is null if and only if its complement is generic.
\item The entire set $\euS$ is generic in $\euS$, and being generic in $\euS$ implies denseness in $\PML$.  (Because if $A$ misses an 
open set in $\PML$, then the closure of its complement has positive measure.)
\item There are natural subsets of $\euS$ that are neither null nor generic.
For instance, suppose that $g\ge 2$, so that $\euS$ has a nontrivial partition into {\em separating} and 
{\em nonseparating} curves.
Each of these subsets is dense in $\PML$, so  neither can be null or generic.
\item Our basic example of a null set in $\euS$ is the 
set of all curves disjoint from some $\alpha$, which is a copy of $\euS(S_\alpha)$ sitting inside
$\euS(S)$.  Its closure in $\PML(S)$ consists of those laminations giving zero weight to $\alpha$, which can be 
identified with a copy of $\PML(S_\alpha)$; since this has lower dimension, it has measure zero.  
\end{itemize}

Because nullness is not preserved by countable unions, the following proposition is less obvious.
It is  proved by Lustig--Moriah \cite[Cor 5.3]{LM2} using techniques from handlebodies, but we include a proof for completeness.  
Our proof is similar to well-known arguments, due to Luo and to Kobayashi \cite{Kobayashi}, showing that the complex of curves has infinite diameter (see the remarks following Proposition 4.6 of \cite{MM1}).

\begin{proposition}
\label{prop:bounded}
Any bounded-diameter subset of the curve graph  $\CC(S)$ is a null subset of $\euS$. 
\end{proposition}

We first prove that chains of disjoint laminations between two curves can be realized with curves.

\begin{definition}
Given a pair of laminations $F,G\in\PML$, define their {\em intersection distance} 
$\di(F,G)$ to be the smallest $n$  such that there exist 
$F=G_0,G_1,\ldots,G_n=G$ with $\I(G_j,G_{j+1})=0$.  
\end{definition}
(We note that $\di(F,G)$ can be infinite; this occurs if at least one of them is filling and they are topologically distinct.)

\begin{lemma} 
\label{lem:intersectiondistance}
$\di(\alpha,\beta)=\dc(\alpha,\beta)$ for all $\alpha,\beta\in\euS$.
\end{lemma}

\begin{proof}
Since simple closed curves are laminations, 
 it is immediate that $\di(\alpha,\beta)\le \dc(\alpha,\beta)$.

For the other direction, we can assume that the laminations are not filling. Then given 
 a lamination $G\in\PML$, let us write $Y_G$ for its supporting subsurface.
Now  $\I(F,G)=0 \implies \I(F,\partial Y_G)=0 \implies \I(\partial Y_F,\partial Y_G)=0$.
But then given a
 minimal-length disjointness path 
$$\alpha - G_1 - G_2 - \cdots - G_{n-1} - \beta,$$
we can realize it by simple closed curves by replacing each $G_i$ with $\partial Y_{G_i}$.
\end{proof}

\begin{proof}[Proof of Proposition~\ref{prop:bounded}]
First given  $\alpha\in\euS$, let $\Sph r\alpha \subset \CC$ denote
the sphere of radius $r$ centered at $\alpha$. 
It is enough to prove the Proposition for each $\Sph r\alpha$.  For the ball of radius $1$, the statement follows since each $\beta$ 
satisfies $i(\alpha,\beta)=0$ and, as we saw above, the set of such $\beta$ has measure $0$ closure in $\PML(S)$. Notice that this closure $\overline{\Sph{1}{\alpha}}$, which we identify with a copy of $\PML(S_\alpha)$, is exactly the set of laminations $F\in \PML(S)$ for which $i(\alpha,F) = 0$.

Now we consider the closure of 
the sphere of radius $r$ and suppose  $G\in \Sphbar r{\alpha}$. Then  $G=\lim_{m\to\infty} \beta^r_m$ where 
$\beta^r_m\in \Sph r\alpha$.  Then for each $m$ there is a path  $\alpha,\beta_m^1,\ldots, \beta_m^r$ in $\CC$,where 
$\beta_m^j\in \Sph j{\alpha}$.  
Passing to subsequences we can assume that for each $j$ the sequence $\beta_m^j$ converges to some 
$G_j\in\Sphbar j{\alpha}$ with $G_r=G$.  
Furthermore since $\I(\beta_m^j,\beta_m^{j+1})=0$ it follows that $$\I(G_j,G_{j+1})=0.$$   
Replacing each $G_j$ with $\partial Y_{G_j}$, as in the proof of Lemma~\ref{lem:intersectiondistance},  we see that $\I(G,\gamma)=0$ for some $\gamma\in \Sph{r-1}{\alpha}$.  
But then $G\in \Sphbar 1{\gamma}$. Thus 
\[  \Sphbar r{\alpha} \subset \bigcup_{\gamma\in\Sph{r-1}{\alpha}}\Sphbar 1{\gamma}.\]
Thus $\Sphbar r{\alpha}$ is a countable union of measure-zero subsets of 
$\PML(S)$, hence has zero measure itself. Therefore $\Sph r\alpha$ is a null set by definition. 
\end{proof}

We wish to define what it means for a property to be generic for pairs of points in $\Sph r\alpha$.  Although $\Sph r\alpha$ is contained in $\euS$, the Lustig--Moriah definition of genericity in $\euS$ does not apply because the set $\Sph r\alpha$ is itself null in $\euS$.  Nonetheless, $\Sphbar 1\alpha=\PML(S_\alpha)$ is a topological sphere in its own right, and thus has its own natural measure class.  Therefore, we may define a subset $E$ of $\Sph 1\alpha$ to be \emph{null} if $\overline{E}$ has measure zero in $\PML(S_\alpha)$; 
this respects the topology of $\PML(S_\alpha)$ sitting inside of $\PML(S)$. We extend this notion of nullness to subsets $E\subset \Sph r\alpha$ of larger spheres in a ``visual'' manner by considering the set of points on the sphere of radius $1$ that are metrically between the center and $E$---these are the points of $\Sph 1\alpha$ that ``see'' the set $E$.

\begin{definition}\label{def:null}
 $E$ is {\em null} in $\Sph r\alpha$ if 
if  $E_1:=\{\gamma \in \Sph 1\alpha : E\cap \Sph{r-1}{\gamma} \neq\emptyset\}$ is null in 
$\Sph 1\alpha \hookrightarrow \PML(S_\alpha)$.
\end{definition}

\begin{remark}  The definition given above is the most restrictive notion of nullness that makes 
use of $\Sph 1\alpha$ as a visual sphere (i.e., that treats the $1$--sphere as the sphere of directions).  
Another possible definition, also natural from the point of view
of Fubini's theorem, would be inductive:  suppose nullness has been defined 
for spheres of radius $1,\ldots,r-1$.
Instead of $E_1$, the full set of points that see $E\subset\Sph r\alpha$, we form the smaller set 
$$E_1'=\{\gamma\in \Sph 1\alpha : E\cap \Sph{r-1}{\gamma} 
~\hbox{\rm is  not null in $\Sph{r-1}{\gamma}$}\}.$$
Then we could declare $E\subset \Sph r \alpha$ to be null in $\Sph r\alpha$  if
 $E_1'$ is null in $\Sph 1\alpha$, completing the inductive definition.
\end{remark}

\begin{example}
To get a feeling for these definitions,  consider the examples of $\R^2$ with the Euclidean 
metric or $\ell^1$ metric, with the Lebesgue measure class on the sphere of radius 1 in each case.
To accord with geometric intuition, we expect arcs to be non-null and points to be null.

\begin{figure}[ht]
\begin{tikzpicture}[scale=3/4]

\filldraw (0,0) circle (0.05);
\draw (0,0) circle (1);
\draw (0,0) circle (2);
\draw [blue,line width=3] (110:2) arc (110:140:2);
\draw [blue,line width=3] (110:1) arc (110:140:1);
\filldraw [red] (40:2) circle (0.05);
\filldraw [red] (40:1) circle (0.05);

\begin{scope}[xshift=5cm]
\filldraw (0,0) circle (0.05);
\draw (-1,0) -- (0,-1) -- (1,0) -- (0,1)-- cycle;
\draw (-2,0) -- (0,-2) -- (2,0) -- (0,2)-- cycle;
\draw [blue, line width=3] (-1.2,.8)--(-.8,1.2);
\draw [blue, line width=3] (-1,0)--(0,1);

\filldraw [red] (1.2,.8) circle (0.05);
\draw [red, line width=3] (0,1)--(1,0);
\end{scope}
\end{tikzpicture}
\caption{In each metric, an arc $E$ and a point $E$ are shown on the sphere of radius two
together with the associated $E_1$ for each.}
\end{figure}
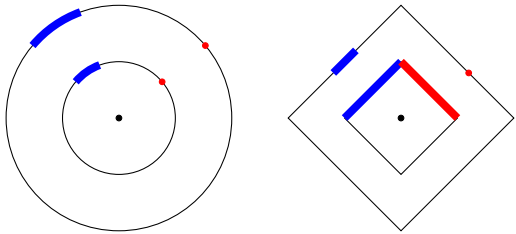

In the Euclidean metric, if $E$ is an arc on the sphere of radius $r$, then $E_1$ is also an arc but $E_1'$ is empty.  
If $E$ is a point, then $E_1$ is a point while $E_1'$ is again empty.
In the $\ell^1$ metric, if $E$ is a nontrivial arc on the sphere of radius $r$, then $E_1$ is a nontrivial arc, and so 
is $E_1'$.  In this setting, however, points in nonaxial directions have a large $E_1$ but an empty $E_1'$.

This means that our visual definition of nullness works intuitively in the $\ell^2$ case (points are null but arcs
are not), but less so in the $\ell^1$ case (where even points are typically non-null).  
The weaker, inductively defined, notion of nullness makes even arcs null in Euclidean space, but on 
the other hand behaves intuitively on $\ell^1$.
This suggests that the visual definition of nullness is better adapted to capturing the geometry of spheres in  certain 
spaces, while the inductive definition would be better adapted to others.  However, being null in our sense
implies nullness in the weaker sense.
\end{example}

Returning to the curve graph:  consider distinct curves
$\beta,\gamma\in \Sph 1\alpha$.  Clearly $\dc(\beta,\gamma)$ is either 1 or 2.  
We can easily see that such pairs generically have distance $2$ because since $\alpha$ and $\beta$ are disjoint, the set of $\gamma\in \Sph{1}{\alpha}$ for which $d_S(\beta,\gamma)=1$ is contained in $\PML(S_{\alpha,\beta})$, which has codimension at least two in $\PML(S_{\alpha})$.

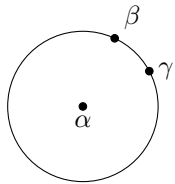
\begin{figure}[ht]
\begin{tikzpicture}
\draw (0,0) circle (1);
\filldraw (65:1) circle (0.05) node [above right] {$\beta$}; 
\filldraw (28:1) circle (0.05) node [right] {$\gamma$}; 
\filldraw (0,0) circle (0.05) node [below] {$\alpha$};
\end{tikzpicture}
\end{figure}

Our main result is that given any limit on their length, paths connecting points on the sphere ``almost surely" 
pass through the center.

\begin{theorem}[Avoiding the center]  For a surface $S$ with $h\ge 4$, consider a point $\alpha\in\CC(S)$.
For $K>0$, let $\short Kr\alpha \subseteq \Sph r\alpha \times \Sph r\alpha$ consist of 
those pairs $(\beta,\gamma)$ that are connected by some 
path of length $\le K$ that does not go through $\alpha$.  Then for any $K$ and $r$, the set $\short Kr\alpha$ is 
null. 
\end{theorem}

\begin{proof}  
For any pair $(\beta,\gamma)\in \short Kr\alpha$, there is a path $\beta=\delta_0, \delta_1,\ldots,\delta_k=\gamma$
in $\CC(S)$ with $k\le K$, and $\delta_i\neq \alpha$ for each $i$.  
Two successive curves $\delta_i$ and $\delta_{i+1}$, since they are disjoint and  intersect $S_\alpha$, 
  project to nonempty sets in $\CC(S_\alpha)$ whose distance from each other is at most $4$; thus  $d_{S_\alpha}(\beta,\gamma)\le 4K$.

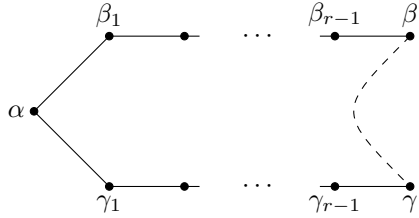
\begin{figure}[ht]
\begin{tikzpicture}

\draw (0,0) node [left] {$\alpha$}--(1,1) node [above] {$\beta_{1}$} --(2,1)--(2.2,1);
\node at (3,1) {$\cdots$};
\draw (3.8,1)--(4,1)node [above] {$\beta_{r-1}$}  -- (5,1) node [above] {$\beta$};
\draw (0,0)--(1,-1)  node [below] {$\gamma_{1}$} --(2,-1)-- (2.2,-1);
\node at (3,-1) {$\cdots$};
\draw (3.8,-1)--(4,-1) node [below] {$\gamma_{r-1}$} -- (5,-1) node [below] {$\gamma$};
\foreach \a in {(0,0),(1,1),(2,1),(4,1),(5,1),(1,-1),(2,-1),(4,-1),(5,-1)}
\filldraw \a circle (0.05);
\draw [dashed] (5,1) .. controls (4,0) .. (5,-1);
\end{tikzpicture}
\end{figure}

Let $\beta_{1},\gamma_{1}$ be any closest points 
 on $\Sph {1}\alpha$ to $\beta,\gamma$, respectively. 
Since we can join $\gamma,\gamma_1$ by a geodesic in $\CC(S)$ that misses $\alpha$,
there is a constant $M=M(S)$ coming from  Masur-Minsky \cite[Thm 3.1]{MM2} 
such that $d_{S_\alpha}(\gamma,\gamma_1)\leq M$.  
By the triangle inequality, 
 $$d_{S_\alpha}(\beta,\gamma_{1})\leq M+4K.$$ 
 For each $\beta\in \Sph r \alpha$, let $E(\beta)=\{\gamma:(\beta,\gamma)\in 
 \short Kr\alpha\}$ and then consider the corresponding 
 $E_1(\beta)=\{\gamma_1 \in \Sph 1\alpha : E(\beta)\cap \Sph{r-1}{\gamma_1} \neq\emptyset\}$.

We have shown  that $E_1(\beta)$ has diameter at most $2M+8K$ by the triangle inequality and is therefore 
null in $\Sph{1}{\alpha}$ by 
Proposition~\ref{prop:bounded}.  Thus $E(\beta)$ is null for all $\beta$, so $\short Kr\alpha$ 
is null.  \end{proof}

\begin{corollary}[Statistical hyperbolicity]
Consider pairs of points $\beta,\gamma\in\Sph r\alpha$ such that $d_S(\beta,\gamma)<2r$.
The set of all such pairs is null.  So for a generic pair of points on $\Sph r\alpha$, the distance is
exactly $2r$.
\end{corollary}

In closing, we note that there is a coarsely well-defined, coarsely Lipschitz
map from Teichm\"uller space to the curve
complex given by taking a short marking at every point.
With respect to natural measures on $\T(S)$, 
generic geodesics make definite progress in the curve graph, as we show in \cite{DDM1}.  
Thus we can see a very loose heuristic for calculating average distances in  $\T(S)$ by projecting to $\CC(S)$
and appealing to the phenomena that we have just demonstrated in the curve graph. Obtaining the precise aymptotics needed for statistical hyperbolicity, as we do in \cite{DDM1}, takes much more work.



\end{document}